\documentclass[10pt]{amsart}
\usepackage{amsmath,amssymb,amsthm} 

\usepackage[ps2pdf=true,colorlinks]{hyperref}
\usepackage[figure,table]{hypcap}
\usepackage[T1]{fontenc}
\hypersetup{
       bookmarksnumbered,
        pdfstartview={FitH},
        citecolor={blue},
        linkcolor={red},
        urlcolor={magenta},
        pdfpagemode={UseOutlines}
}
\makeatletter
\newcommand\org@hypertarget{}
\let\org@hypertarget\hypertarget
\renewcommand\hypertarget[2]{%
  \Hy@raisedlink{\org@hypertarget{#1}{}}#2%
}
\makeatother

\theoremstyle{definition}
\newtheorem{definition}{Definition}[section]
\theoremstyle{plain}
\newtheorem{theorem}[definition]{Theorem}
\newtheorem{proposition}[definition]{Proposition}
\newtheorem{corollary}[definition]{Corollary}

\numberwithin{equation}{section}

\title{The Similarity Problem for $\mathcal{Z}$-stable $\mathrm{C}^{*}$-algebras}
\author{Miroslava Johanesov\'a}
\address{School of Mathematical Sciences, University of Nottingham, University Park, \indent Nottingham, NG7 2RD, United Kingdom}
\email{pmxmj1@nottingham.ac.uk}

\author{Wilhelm Winter}
\email{wilhelm.winter@nottingham.ac.uk}
\thanks{The second named author is partially supported by EPSRC First Grant EP/G014019/1.}

\date{December 10, 2010}

\begin{document}

\maketitle

\begin{abstract}
We show that the tensor product of two unital C$^{*}$-algebras, one of which is nuclear and admits a unital $^{*}$-homomorphism from (the building blocks of) the Jiang-Su algebra, has Kadison's similarity property. As a consequence, we obtain that a unital C$^{*}$-algebra which absorbs the Jiang-Su algebra tensorially also has this property.
\end{abstract}

\section{Introduction}

Given a C$^{*}$-algebra $\mathcal{A}$ and a $^{*}$-homomorphism $\rho\colon\mathcal{A}\rightarrow\mathcal{B}(\mathcal{H})$ ($\mathcal{H}$ a Hilbert space), the canonical way of constructing a (non-selfadjoint) homomorphism $\pi\colon\mathcal{A}\rightarrow\mathcal{B}(\mathcal{H})$ is by setting $\pi(\cdot):=S^{-1}\rho(\cdot)S$, where $S\colon\mathcal{H}\rightarrow\mathcal{H}$ is a bounded invertible operator. Moreover, since $\|\pi\|\leq\|S\|\|S^{-1}\|$, $\pi$ is necessarily bounded. R. Kadison in \cite{Ka} asked the natural question of whether all bounded homomorphisms are of this form. A C$^{*}$-algebra $\mathcal{A}$ is said to have  the \emph{similarity property} if every bounded homomorphism $\pi\colon\mathcal{A}\rightarrow\mathcal{B}(\mathcal{H})$ is similar to a $^{*}$-homomorphism, that is, if there exists an invertible operator $S\in\mathcal{B}(\mathcal{H})$ such that $S\pi(\cdot)S^{-1}$ is a $^{*}$-homomorphism. Then the question posed by Kadison is equivalent to asking whether every C$^{*}$-algebra has the similarity property.\\
\indent Kadison's question in its full generality remains open but several partial results are known. In particular, U. Haagerup \cite{Ha} proved  that a bounded homomorphism $\pi$ is similar to a $^{*}$-homomorphism if and only if it is completely bounded. Recall that by definition
\begin{equation}\label{eq:cb}
\|\pi\|_{\mathrm{cb}}:=\sup_{n\in\mathbb{N}}\|\pi^{(n)}\|,
\end{equation}
where $\pi^{(n)}\colon M_{n}(\mathcal{A})\rightarrow M_{n}(\mathcal{B}(\mathcal{H}))$ is given by $\pi^{(n)}:=\mathrm{id}_{M_{n}}\otimes\pi$ and we call $\pi$ \emph{completely bounded} if $\|\pi\|_{\mathrm{cb}}$ is finite. Moreover, Haagerup showed that if $\pi$ is completely bounded, then
\begin{eqnarray*}
\|\pi\|_{\mathrm{cb}}=\inf\{\|S\|\|S^{-1}\| & : & S\in\mathcal{B}(\mathcal{H})\ \text{invertible such that} \\
&& \ S\pi(\cdot)S^{-1}\ \text{is a $^{*}$-homomorphism}\}
\end{eqnarray*}
and the infimum is attained.\\
\indent In \cite{Pi1}, G. Pisier proved that if $\mathcal{A}$ has the similarity property, then there are
  a number $\alpha$ and a constant $C$ (depending on $\alpha$) such that any bounded homomorphism $\pi\colon\mathcal{A}\rightarrow\mathcal{B}(\mathcal{H})$ satisfies
\begin{equation*}
\|\pi\|_{\mathrm{cb}}\leq C\|\pi\|^{\alpha}.
\end{equation*}
Moreover, the smallest $\alpha$ with this property exists and is an integer; it is called the \emph{similarity degree of $\mathcal{A}$}, denoted $d(\mathcal{A})$. It now follows that a C$^{*}$-algebra has the similarity property if and only if its similarity degree is finite.\\
\indent Although Kadison's problem is still open, the similarity degrees of some specific classes of C$^{*}$-algebras are already known. Some of the most important of these we list below, together with their similarity degrees (or the upper bounds thereof) in the case of infinite dimensional C$^{*}$-algebras. It is, of course, well known that for the finite dimensional case, the similarity degree is 1.
\begin{itemize}
 \item[(i)] $\mathcal{A}$ is nuclear if and only if $d(\mathcal{A})= 2$ (\cite{Bu}, \cite{Ch3}, \cite{Pi4}).
 \item[(ii)] If $\mathcal{A}=\mathcal{B(H)}$, then $d(\mathcal{A})=3$ (\cite{Ha}, \cite{Pi1}).
 \item[(iii)] If $\mathcal{A}$ is unital, then $d(\mathcal{A}\otimes\mathcal{K}(\mathcal{H}))\leq 3$ (\cite{Ha}, \cite{Pi3}).
 \item[(iv)] If $\mathcal{A}$ is a $II_{1}$ factor with property $\Gamma$, then $d(\mathcal{A})=3$ (\cite{Ch2}, \cite{Pi3}).
 \item[(v)] If $\mathcal{A}$ and $\mathcal{B}$ are unital and $\mathcal{B}$ contains unitally a matrix algebra of any order, then $d(\mathcal{A}\otimes\mathcal{B})\leq 5$ (\cite{Po}).
\item[(vi)] If $\mathcal{A}$ is unital, separable and approximately divisible, then $d(\mathcal{A})\leq 5$ (\cite{LiSh}).
\end{itemize}

\indent In this paper we show that if $\mathcal{A}$ and $\mathcal{B}$ are separable unital C$^{*}$-algebras, such that $\mathcal{B}$ is nuclear and admits a unital $^{*}$-homomorphism from (the building blocks of) the Jiang-Su algebra $\mathcal{Z}$ (we refer the reader to \cite{JiSu} and \cite{RoWi} for the original construction and various characterisations of the Jiang-Su algebra), then $d(\mathcal{A}\otimes\mathcal{B})\leq 5$. This generalises the result obtained by F. Pop in \cite{Po}, outlined in (v) above. We note that the unital embeddability of the Jiang-Su algebra cannot be automatically assumed. In fact, there is an example in \cite{DaHiToWi} of a unital, simple, nuclear, infinite dimensional C$^{*}$-algebra which does not admit such an embedding.\\
\indent As a corollary, we obtain that the similarity degree of a separable unital C$^{*}$-algebra which absorbs $\mathcal{Z}$ tensorially (a so called $\mathcal{Z}$-stable C$^{*}$-algebra) is at most 5 and thus a unital $\mathcal{Z}$-stable C$^{*}$-algebra has Kadison's similarity property. Since separable approximately divisible C$^{*}$-algebras are $\mathcal{Z}$-stable \cite{ToWi}, this generalises  the result of W. Li and J. Shen \cite{LiSh} mentioned in (vi) above.

\indent The Jiang-Su algebra and the notion of $\mathcal{Z}$-stability have come to the fore of the classification and structure theory of separable nuclear C$^{*}$-algebras in recent years; in this context, $\mathcal{Z}$-stability is closely related to finite topological dimension and often ensures classifiability, cf.\ \cite{Winter:dr-Z-stable}. While it is easy to force a separable C$^{*}$-algebra to be $\mathcal{Z}$-stable (any one of the form $\mathcal{A} \otimes \mathcal{Z}$ will do), at this point little is known about the meaning of $\mathcal{Z}$-stability for arbitrary separable C$^{*}$-algebras, and how often it occurs. For example, it is an open question whether the reduced C$^{*}$-algebra of the free group with two generators is $\mathcal{Z}$-stable; this does not seem entirely unlikely as $\mathrm{C}^{*}(\mathbb{F}_{2})$ has many central sequences, cf.\ \cite{Phi:central-sequences}. It is also open whether $\mathrm{C}^{*}_{r}(\mathbb{F}_{2})$ has the similarity property, and it is an interesting possibility that both questions are equivalent. Of course, in this note we do not make  substantial progress on either of these problems, but at least our result says that a positive answer to the first one would imply the second. We regard our paper as a first attempt to exploit $\mathcal{Z}$-stability also outside the realm of nuclear C$^{*}$-algebras.

\section{Notation and Preliminaries}

In this section we introduce some notation, and state some preliminary results about $C(X)$-algebras and completely bounded maps.

For each natural number $n$, $M_{n}$ will denote the C$^{*}$-algebra of $n\times n$ matrices over the complex numbers with unit $1_{M_{n}}$. If $\mathcal{A}$ is a C$^{*}$-algebra, $M_{n}(\mathcal{A})$ will denote the C$^{*}$-algebra of $n\times n$ matrices with elements in $\mathcal{A}$ and we identify $M_{n}(\mathcal{A})$ with $M_{n}\otimes\mathcal{A}$ in the usual way. For a pair of C$^{*}$-algebras $\mathcal{A}$ and $\mathcal{B}$, we denote by $\mathcal{A}\odot\mathcal{B}$, $\mathcal{A}\otimes\mathcal{B}$ and $\mathcal{A}\otimes_{\mathrm{max}}\mathcal{B}$ their algebraic, minimal and maximal tensor product respectively.

For natural numbers $p$ and $q$, we set
\begin{equation*}
Z_{p,q}:=\{f\in C([0,1], M_{p}\otimes M_{q}):f(0)\in M_{p}\otimes 1_{M_{q}}, f(1)\in 1_{M_{p}}\otimes M_{q}\},
\end{equation*}
to be the so-called \emph{dimension drop $\mathrm{C}^{*}$-algebra}. If $p$ and $q$ are relatively prime, then $Z_{p,q}$ is said to be \emph{prime}.

The concept of $C_{0}(X)$-algebras, which may be considered as generalised fields of C$^{*}$-algebras parametrised by a locally compact space $X$, was introduced by G. Kasparov in \cite{Kas}. In this paper, we only require $X$ to be compact (in fact the unit interval), in which case it is standard to refer to such objects as $C(X)$-algebras. If we also restrict to the case of unital C$^{*}$-algebras, we can state the following definition:

\begin{definition}
A unital C$^{*}$-algebra $\mathcal{A}$ is called a $C([0,1])$-algebra if there exists a unital $^{*}$-homomorphism $\mu : C([0,1])\rightarrow Z(\mathcal{A})$, where $Z(\mathcal{A})$ denotes the centre of $\mathcal{A}$. The map $\mu$ is called the \emph{structure map}.
\end{definition}
\noindent If $\mathcal{A}$ is a $C([0,1])$-algebra with structure map $\mu$ and $t\in [0,1]$, then
\begin{equation*}
J_{t}:=\mu(\{f\in C([0,1])\colon f(t)=0\})\cdot\mathcal{A}
\end{equation*}
is a closed two-sided ideal of $\mathcal{A}$ and
\begin{equation*}
\mathcal{A}_{t}:=\mathcal{A}/J_{t}
\end{equation*}
is called the \emph{fibre of $\mathcal{A}$ at $t$}. Moreover, for any $a\in\mathcal{A}$ we have that
\begin{equation}\label{eq:fibre}
\|a\|=\sup_{t\in [0,1]}\|a_{t}\|,
\end{equation}
where $a_{t}:=\pi_{t}(a)$ and $\pi_{t}:\mathcal{A}\rightarrow\mathcal{A}_{t}$ is the quotient map.

We note that $Z_{p,q}$ is a $C([0,1])$-algebra in a canonical way, with fibres $M_{p}$ at the left end-point, $M_{pq}$ in $(0,1)$ and $M_{q}$ at the right end-point.\\

Finally, we recall the following fact about completely bounded maps:

\begin{theorem}\cite[Theorem 12.3]{Pa}\label{paulsen}
For $i=1, 2$ let $\mathcal{A}_{i}$ and $\mathcal{B}_{i}$ be unital $\mathrm{C}^{*}$-algebras and let $\pi_{i}\colon\mathcal{A}_{i}\rightarrow\mathcal{B}_{i}$ be completely bounded. Then the linear map $\pi_{1}\odot\pi_{2}\colon\mathcal{A}_{1}\odot\mathcal{A}_{2}\rightarrow\mathcal{B}_{1}\odot\mathcal{B}_{2}$, given by $(\pi_{1}\odot\pi_{2})(a_{1}\otimes a_{2})=\pi_{1}(a_{1})\otimes\pi_{2}(a_{2})$, defines a completely bounded map $\pi_{1}\otimes\pi_{2}\colon\mathcal{A}_{1}\otimes\mathcal{A}_{2}\rightarrow\mathcal{B}_{1}\otimes\mathcal{B}_{2}$, with
\begin{equation*}
\|\pi_{1}\otimes\pi_{2}\|_{\mathrm{cb}}=\|\pi_{1}\|_{\mathrm{cb}}\|\pi_{2}\|_{\mathrm{cb}}.
\end{equation*}
\end{theorem}

\vspace*{0.3cm}
We isolate the following observation for convenience:

\begin{proposition}\label{lemma}
Let $\mathcal{A}, \mathcal{B}$ and $\mathcal{C}$ be unital $\mathrm{C}^{*}$-algebras and let $\rho\colon\mathcal{A}\otimes\mathcal{B}\rightarrow\mathcal{C}$ be a unital homomorphism, such that $\rho|_{1_{\mathcal{A}}\otimes\mathcal{B}}$ is a $^{*}$-homomorphism. If $q\in\mathcal{B}$ is a projection such that $q^{\bot}:=1_{\mathcal{B}}-q\prec q$ \emph{(}in the Murray-von Neumann sense\emph{)}, then for $x\in\mathcal{A}\otimes1_{\mathcal{B}}$
\begin{equation*}
\|\rho(x)\|=\|\rho(x(1_{\mathcal{A}}\otimes q))\|.
\end{equation*}
\end{proposition}

\begin{proof}
Let $\mathcal{D}\subset\mathcal{C}$ be the $\mathrm{C}^{*}$-algebra generated by $\rho(\mathcal{A}\otimes1_{\mathcal{B}})$ and put $e:=\rho (1_{\mathcal{A}}\otimes q)\in \mathcal{C}$. Note that $e$ is a projection and it commutes with elements of $\mathcal{D}$. Now set $e^{\bot}:=1_{\mathcal{D}}-e = \rho(1_{\mathcal{A}}\otimes q^{\bot})\in\mathcal{C}$ and it too commutes with elements of $\mathcal{D}$.\\
\indent Next define $\Phi\colon\mathcal{D}\rightarrow\overline{e\mathcal{D}e}\oplus\overline{e^{\bot}\mathcal{D}e^{\bot}}$ by $d\mapsto ede+e^{\bot}de^{\bot}$. It then follows that for $x\in\mathcal{A}\otimes1_{\mathcal{B}}$
\begin{equation*}
\|\rho(x)\|=\max\{\|e\rho(x)e\|,\|e^{\bot}\rho(x)e^{\bot}\|\}.
\end{equation*}
But since $q^{\bot}\prec q$, there exists a partial isometry $s\in\mathcal{B}$ such that
\begin{eqnarray*}
s^{*}s=q^{\bot} & \text{and} & ss^{*}\leq q.
\end{eqnarray*}
Therefore,
\begin{align*}
\|e^{\bot}\rho(x)e^{\bot}\| & =  \|\rho(1_{\mathcal{A}}\otimes q^{\bot})\rho(x)\rho(1_{\mathcal{A}}\otimes q^{\bot})\|\\
                                            & =  \|\rho(1_{\mathcal{A}}\otimes s^{*}s)\rho(x)\rho(1_{\mathcal{A}}\otimes s^{*}s)\|\\
                                            & =  \|\rho(1_{\mathcal{A}}\otimes s^{*})\rho(1_{\mathcal{A}}\otimes s)\rho(x)\rho(1_{\mathcal{A}}\otimes s^{*})\rho(1_{\mathcal{A}}\otimes s)\|\\
                                            & \leq \|\rho(1_{\mathcal{A}}\otimes s)\rho(x)\rho(1_{\mathcal{A}}\otimes s^{*})\|,
\end{align*}
as $\|\rho(1_{\mathcal{A}}\otimes s)\|\leq\|\rho\|\leq1$. Since $x$ commutes with $1_{\mathcal{A}}\otimes s^{*}$, we further have
\begin{align*}
 \|e^{\bot}\rho(x)e^{\bot}\| & \leq \|\rho(1_{\mathcal{A}}\otimes ss^{*})\rho(x)\|\\
                                             & = \|\rho(x)^{*}\rho(1_{\mathcal{A}}\otimes ss^{*})\rho(x)\|^{\frac{1}{2}}\\
                                             & \leq \|\rho(x)^{*}e\rho(x)\|^{\frac{1}{2}}\\
                                             & = \|e\rho(x)e\|.
\end{align*}
Thus,
\begin{equation*}
\|\rho(x)\| \leq \|e\rho(x)e\| \leq \|\rho(x)e\| \leq \|\rho(x(1_{\mathcal{A}}\otimes q))\| \leq \|\rho(x)\|.\qedhere
\end{equation*}
\end{proof}

\section{The Main Results}

\begin{theorem}\label{theorem}
Let $\mathcal{A}$ and $\mathcal{B}$ be unital $\mathrm{C}^{*}$-algebras such that, for all relatively prime natural numbers $p$ and $q$, there exists a unital $^{*}$-homomorphism $Z_{p,q}\rightarrow\mathcal{B}$. If $\rho\colon \mathcal{A}\otimes\mathcal{B}\rightarrow\mathcal{B}(\mathcal{H})$ is a unital bounded homomorphism such that $\rho_{\mathcal{B}}:=\rho|_{1_{\mathcal{A}}\otimes\mathcal{B}}$ is a $^{*}$-homomorphism, then $\rho_{\mathcal{A}}:=\rho|_{\mathcal{A}\otimes 1_{\mathcal{B}}}$ is completely bounded and
\begin{equation}\label{eq:theorem}
\|\rho_{\mathcal{A}}\|_{\mathrm{cb}}\leq\|\rho\|.
\end{equation}
\end{theorem}

\begin{proof}
For a given $n\in\mathbb{N}$ choose $n\leq p, q\in\mathbb{N}$ relatively prime. Let $\mathcal{D}\subset\mathcal{B}(\mathcal{H})$ be the C$^{*}$-algebra generated by $\rho(\mathcal{A}\otimes Z_{p,q})$ and put $\bar{\rho}:=\rho|_{\mathcal{A}\otimes Z_{p,q}}$. Then $\mathcal{D}$ is a $C([0, 1])$-algebra and $\bar{\rho}$ induces a well-defined map $\bar{\rho}_{t}\colon (\mathcal{A}\otimes Z_{p,q})_{t}\rightarrow\mathcal{D}_{t}$ from the fibre of $\mathcal{A}\otimes Z_{p,q}$ at $t$ to the fibre of $\mathcal{D}$ at $t$, $t\in[0, 1]$. Note also that $(\mathcal{A}\otimes Z_{p,q})_{t}=\mathcal{A}\otimes M_{r_{t}}$, where $r_{t}$ equals $p$ at the left end-point, $pq$ in $(0, 1)$ and $q$ at the right end-point. It then follows from \eqref{eq:fibre} that
\begin{equation*}
\|\bar{\rho}\|=\sup_{t\in [0,1]}\|\bar{\rho}_{t}\|
\end{equation*}
and, upon amplifying,
\begin{equation}\label{eq:star}
\|\bar{\rho}^{(n)}\|=\sup_{t\in [0,1]}\|\bar{\rho}_{t}^{(n)}\|.
\end{equation}

Now for each $t\in [0,1]$ let $\nu_{t}\colon M_{n}\hookrightarrow M_{r_{t}}$ be an embedding of $M_{n}$ into $M_{r_{t}}$, such that $1_{M_{r_{t}}}-\nu_{t}(1_{M_{n}})\prec\nu_{t}(1_{M_{n}})$. If for $x_{t}\in M_{n}\otimes\mathcal{A}\otimes 1_{M_{r_{t}}}$ with $\|x_{t}\|\leq 1$ we set
\begin{equation*}
y_{t} := x_{t}\cdot 1_{M_{n}\otimes\mathcal{A}}\otimes\nu_{t}(1_{M_{n}})\ \ \ \ \ \ \text{and}\ \ \ \ \ \ y^{\bot}_{t} := x_{t}\cdot 1_{M_{n}\otimes\mathcal{A}}\otimes (1_{M_{r_{t}}}-\nu_{t}(1_{M_{n}})),
\end{equation*}
then $x_{t} = y_{t} + y^{\bot}_{t}$ and Proposition \ref{lemma} 
gives
\begin{equation*}
\|\bar{\rho}_{t}^{(n)}(x_{t})\| = \|\bar{\rho}_{t}^{(n)}(y_{t})\|.
\end{equation*}

Next we use a unitary $u\in M_{n}\otimes 1_{\mathcal{A}}\otimes\nu_{t}(M_{n})$ to flip the two copies of $M_{n}$, so that
\begin{equation*}
z_{t}: = u y_{t} u^{*} \in 1_{M_{n}}\otimes\mathcal{A}\otimes \nu_{t}(M_{n});
\end{equation*}
we get
\begin{equation*}
\|\bar{\rho}_{t}^{(n)}(y_{t})\| = \|\bar{\rho}_{t}^{(n)}(u^{*})\bar{\rho}_{t}^{(n)}(z_{t})\bar{\rho}_{t}^{(n)}(u)\| =\|\bar{\rho}_{t}^{(n)}(z_{t})\|.
\end{equation*}
This follows from the fact that $\bar{\rho}_{t}^{(n)}:=\mathrm{id}_{M_{n}}\otimes\bar{\rho}_{t}$ and $\bar{\rho}_{t}$ restricted to $1_{\mathcal{A}}\otimes\nu_{t}(M_{n})$ is a $^{*}$-homomorphism. Thus $\bar{\rho}_{t}^{(n)}$ restricted to $M_{n}\otimes1_{\mathcal{A}}\otimes\nu_{t}(M_{n})$ is also a $^{*}$-homomorphism (as a tensor product of two $^{*}$-homomorphisms), which preserves unitaries. Writing $z_{t}=1_{M_{n}}\otimes z'_{t}$ with $z'_{t}\in \mathcal{A}\otimes \nu_{t}(M_{n})$, we get for all $n\in\mathbb{N}$ and all $t\in [0,1]$
\begin{equation*}
\|\bar{\rho}_{t}^{(n)}(x_{t})\| = \|\bar{\rho}_{t}(z'_{t})\| \leq \|\bar{\rho}_{t}\|, 
\end{equation*}
since $\|z'_{t}\|=\|x_{t}\|\leq 1$. This together with \eqref{eq:star} gives
\begin{equation*}
\|\rho^{(n)}_{\mathcal{A}}\| = \|\bar{\rho}^{(n)}_{\mathcal{A}}\| \leq \|\bar{\rho}\|  \leq \|\rho\|.
\end{equation*}
The result now follows by taking suprema over all $n\in\mathbb{N}$ and applying definition \eqref{eq:cb}.
\end{proof}

\begin{corollary}\label{corollary}
Let $\mathcal{A}$ and $\mathcal{B}$ be unital $\mathrm{C}^{*}$-algebras, such that $\mathcal{B}$ is nuclear and for all relatively prime natural numbers $p$ and $q$, there exists a unital $^{*}$-homomorphism $Z_{p,q}\rightarrow\mathcal{B}$ (for example $\mathcal{B}$ could contain a unital copy of $\mathcal{Z}$). Then
\begin{equation*}
d(\mathcal{A}\otimes\mathcal{B})\leq 5.
\end{equation*}
\end{corollary}

\begin{proof}
Let $\pi\colon\mathcal{A}\otimes\mathcal{B}\rightarrow\mathcal{B}(\mathcal{H})$ be a unital bounded homomorphism; we denote by $\pi_{\mathcal{A}}$ and $\pi_{\mathcal{B}}$ the restrictions of $\pi$ to $\mathcal{A}\otimes 1_{\mathcal{B}}$ and $1_{\mathcal{A}}\otimes\mathcal{B}$ respectively. By nuclearity of $\mathcal{B}$, from \cite{Ha}, there exists an invertible operator $S\in\mathcal{B}(\mathcal{H})$ such that $\rho_{\mathcal{B}}:=S\pi_{\mathcal{B}}(\cdot)S^{-1}$ is a $^{*}$-homomorphism. Moreover, $S$ can be chosen in such a way that $\|S\|\|S^{-1}\|\leq\|\pi\|^{2}$. We set
\begin{equation*}
\rho_{\mathcal{A}}:=S\pi_{\mathcal{A}}(\cdot)S^{-1}\ \ \ \ \ \ \text{and}\ \ \ \ \ \ \rho:=S\pi(\cdot)S^{-1}.
\end{equation*}

Next, since $Z_{p,q}$ maps to $\mathcal{B}$ unitally for all relatively prime natural numbers $p$ and $q$, Theorem \ref{theorem} implies that $\rho_{\mathcal{A}}$ is completely bounded and the norm estimate \eqref{eq:theorem} holds. Therefore, by Theorem \ref{paulsen} the induced map $\bar{\rho}\colon\mathcal{A}\otimes\mathcal{B}\rightarrow\mathrm{C}^{*}(\rho_{\mathcal{A}}(\mathcal{A}))\otimes\rho_{\mathcal{B}}(\mathcal{B})$ is completely bounded and
\begin{equation*}
\|\bar{\rho}\|_{\mathrm{cb}}=\|\rho_{\mathcal{A}}\|_{\mathrm{cb}}\cdot\|\rho_{\mathcal{B}}\|_{\mathrm{cb}}.
\end{equation*}
Using the fact that $\rho_{\mathcal{B}}$ is a $^{*}$-homomorphism, we further get
\begin{equation}\label{eq:rho}
\|\bar{\rho}\|_{\mathrm{cb}}\leq\|\rho_{\mathcal{A}}\|_{\mathrm{cb}}\stackrel{\eqref{eq:theorem}}{\leq}\|\rho\|.
\end{equation}
By nuclearity of $\rho_{\mathcal{B}}(\mathcal{B})$ we have that
\begin{equation*}
\mathrm{C}^{*}(\rho_{\mathcal{A}}(\mathcal{A}))\otimes\rho_{\mathcal{B}}(\mathcal{B}) = \mathrm{C}^{*}(\rho_{\mathcal{A}}(\mathcal{A}))\otimes_{\text{max}}\rho_{\mathcal{B}}(\mathcal{B})
\end{equation*}
and by the universal property for maximal tensor products there exists a unique $^{*}$-homomorphism $\mathrm{C}^{*}(\rho_{\mathcal{A}}(\mathcal{A}))\otimes_{\text{max}}\rho_{\mathcal{B}}(\mathcal{B})\rightarrow\mathcal{B}(\mathcal{H})$, which when composed with $\bar{\rho}$ agrees with $\rho$ on elementary tensors and hence, by continuity, everywhere. Thus,
\begin{equation*}
\|\rho\|_{\mathrm{cb}}\leq\|\bar{\rho}\|_{\mathrm{cb}}\stackrel{\eqref{eq:rho}}{\leq}\|\rho\|.
\end{equation*}
Whence, putting everything together
\begin{equation*}
\|\pi\|_{\mathrm{cb}} \leq \|S\|\|S^{-1}\|\|\rho\|_{\mathrm{cb}} \leq \|\pi\|^{2}\|\rho\| \leq \|\pi\|^{5}.\qedhere
\end{equation*}
\end{proof}

\vspace*{0.3cm}
Taking $\mathcal{Z}$ in place of $\mathcal{B}$ in Corollary \ref{corollary}, we immediately have the following:
\begin{corollary}
Let $\mathcal{A}$ be a unital $\mathcal{Z}$-stable $\mathrm{C}^{*}$-algebra. Then $d(\mathcal{A})\leq 5$.
\end{corollary}


\end{document}